\documentclass[oneside,a4paper,11pt]{amsart} % Temporarily "oneside" to have todonotes always on the right
\usepackage{amssymb,verbatim,mathabx}
\usepackage{bbm}

\usepackage[T1]{fontenc}
\usepackage[english]{babel}

\usepackage[textsize=footnotesize,textwidth=20ex,colorinlistoftodos]{todonotes}

\usepackage{hyperref}
\hypersetup{
    colorlinks=true,
    linkcolor=blue,
    filecolor=magenta, 
    citecolor=red,     
    urlcolor=cyan,
}

\usepackage[capitalize]{cleveref}
    \crefname{enumi}{}{}
    \Crefname{enumi}{Item}{Items}
    \crefname{equation}{}{}
    \Crefname{equation}{Equation}{Equations}

\newtheorem{proposition}{Proposition}[section]
\newtheorem{lemma}[proposition]{Lemma}
\newtheorem{corollary}[proposition]{Corollary}
\newtheorem{theorem}[proposition]{Theorem}

\theoremstyle{definition}

\newtheorem{remark}[proposition]{Remark}

\newcommand{\Rad}{{\rm Rad}}

\begin{document}
\title[On orthogonal graphs and their automorphisms]{On orthogonal graphs and their automorphisms}
\author{Hans Cuypers}

\begin{abstract}
We provide a characterization of the  connected subgraphs of the graphs with vertex set the non-isotropic points in a quadratic space $(V,Q)$, 
two points adjacent if and only if they span a tangent line.
Here $(V,Q)$ is a   quadratic space $V$ over a finite field $\mathbb{F}_q$ of order $q$, where $q>3$ is odd,
equipped with a non-degenerate quadratic form $Q$.
The local structure of the graph (i.e. the graph induced on the neighbors of a point) determines the structure of the full graph. 
This characterization helps to determine the automorphism group of these graphs.
\end{abstract}

\maketitle

\section{Introduction}
\label{sect:intro}

Let $\mathbb{F}_q$ be a finite field of order $q$ with $q$ odd.
By $(V,Q)$ we denote the quadratic space consisting of a vector space $V$ over $\mathbb{F}_q$ equipped with a quadratic form $Q$,
which we assume to be non-degenerate.  
Non-degenerate means the only vector  $v$  in $V$ which does satisfy that $Q(v)=0$ and for all $u\in V$ with  
$Q(u)=0$ that also $Q(v+u)=0$ is the vector $v=0$. 

An \emph{isotropic} subspace $W$ of $V$ is a linear subspace $W$ of $V$ with $Q_{|W}=0$.
The \emph{polar space} of $(V,Q)$ is the geometry of all isotropic subspaces of $V$.
It has \emph{rank} equal to the dimension of the largest isotropic subspaces of $V$. (Here we assume the rank to be in $\mathbb{N}\cup\{\infty\}$.)
Since all isotropic spaces $W$ of $V$ of dimension at least $2$ are determined by their isotropic lines,
we also can, if the rank is at least $2$, consider this polar space to be the space of all isotropic points and lines of $V$.
(Here we use the projective language, a point is a $1$-space, a line a $2$-space of $V$.)

The graphs and geometries on the \emph{non-isotropic} points are a bit more complicated. 
We can  consider the  graphs with as vertices the points $\langle v\rangle$ where $v\in V\setminus \{0\}$ with $Q(v)\neq 0$ and two such points $\langle v\rangle$ and 
$\langle w\rangle$ adjacent if and only if $\langle v,w\rangle$ is a tangent line,
i.e.,  $\langle v,w\rangle$ contains a unique point $\langle u\rangle$ with $Q(u)=0$ but $u\neq 0$.
Such graph is not connected, it consists of at least two components. We can fix such a connected part and denote it by $\Gamma_Q$. See \cref{sec:orthogonal}.

We characterize the graph $\Gamma_Q$ by its local structure.

Let $\Gamma$ and $\Delta$ be a graph. Then we say $\Gamma$ is \emph{locally isomorphic} to $\Delta$ if for each vertex $v$ of $\Gamma$ there is a vertex $v'$ of $\Delta$ such that
$\Gamma_v$ and $\Delta_{v'}$ are isomorphic. Here, for any graph $\Delta$ with vertex $v$ the graph  $\Delta_v$ is the subgraph
of $\Delta$ on the vertices which are equal or adjacent to $v$.

Our main result is the following local characterization of the graphs $\Gamma_Q$:

\begin{theorem}\label{main}
Suppose $(V,Q)$ is a non-degenerate quadratic space over the field $\mathbb{F}_q$ with $q>3$ odd. 
Fix a graph $\Gamma_Q$ and assume the rank of the orthogonal polar space of $v^\perp$ is at least $2$ for all $v\in \Gamma_Q$.

Let $\Gamma$ be a  connected graph which is locally isomorphic to $\Gamma_Q$,  then $\Gamma$
is isomorphic to $\Gamma_Q$.
\end{theorem}

This local characterization is obtained by showing that $\Gamma$ does contain a set of lines and planes which correspond to the tangent lines and planes in $\Gamma_Q$.
The geometry thus obtained has been characterized in \cite{cuypers, cuyperspasini} and we can identify it with a tangent geometry of $(V,Q)$, which has $\Gamma_Q$ as collinearity graph .

Such local recognition results have already been studied in the nineties. Indeed, the graphs we consider are related to the ones studied in \cite{cohen}. However, recent results in for example \cite{Ihringer} or \cite{Mitrovic}
show new interest in local recognition of graphs.

\bigskip

Until now, the automorphism group of the graph $\Gamma_Q$ was unknown, see  \cite{srg}.
Of course it contains the group $P\Gamma O(V,Q)$, but it might be bigger.
Our  approach to a proof of \cref{main}, however, shows the automorphism group of $\Gamma_Q$ to be a subgroup of the automorphism group
of the polar space associated to $(V,Q)$. Our approach provides the following result.

\begin{theorem}\label{aut}
Suppose $(V,Q)$ is a non-degenerate quadratic space over a field $\mathbb{F}_q$ with $q>3$ odd.
Fix a connected graph $\Gamma_Q$ and assume the rank of the orthogonal polar space on $v^\perp$ is at least $2$ for all $v\in \Gamma_Q$.

Then the  automorphism group of the graph $\Gamma_Q$ is isomorphic to $P\Gamma O(V,Q)$.
\end{theorem}

Indeed, we deduce that automorphisms of $\Gamma_Q$ map tangents of the graph  to tangents. This implies that the automorphism group of $\Gamma_Q$ is also an automorphism group of the polar space of $(V,Q)$. 
As the automorphism group of this polar space equals the group $P\Gamma O(V,Q)$, this proves $P\Gamma O(V,Q)$ to be the automorphism group of $\Gamma_Q$.

\section{Orthogonal graphs}
\label{sec:orthogonal}

Assume $V$ is a  vector space over the finite field $\mathbb{F}_q$ with $q$ elements, $q$ odd, equipped with a quadratic form $Q$.
The elements of dimension one and two of the projective geometry $\mathbb{P}(V)$ will be called \emph{points} and \emph{lines}, respectively.

By $f$ we denote the symmetric bilinear form on $V$ defined by
$$f(v,w)=Q(v+w)-Q(v)-Q(w)$$
for all $v,w\in V$.
Furthermore $\perp$ denotes \emph{orthogonality}, i.e. for subspaces $U$ and $W$ of $V$ we write $U\perp W$
if $f(u,w)=0$ for all $u\in U$ and $w\in W$.
If $U$ is a subset of points, $U^{\perp}$ denotes the subspace $\{v\in V\mid v\perp u$ for all $u\in U\}$.

The \emph{radical} of $f$ is $\Rad(f)=\{v\in V\mid f(v,w)=0$ for all $w\in V\}$ and the \emph{radical} of $Q$ is
$\Rad(Q)=\{v\in \Rad(f)\mid Q(v)=0\}$. 
We call $Q$ \emph{non-degenerate} if and only if $\Rad(Q)=\{0\}$.

A subspace $U\leq V$ is called \emph{isotropic} if $Q_{|U}=0$.
The \emph{rank} of $Q$ is the dimension of a maximal isotropic subspace.
We assume that rank to be a non-negative integer or $\infty$.

A subspace $U$ is called \emph{anisotropic}, if $Q(u)=0$ for $u\in U$ implies $u=0$.

These notions carry over to the projective space $\mathbb{P}(V)$.
So, an isotropic point is a $1$-space of $V$ which is isotropic.
A point $p\in \mathbb{P}(V)$ is called \emph{anisotropic} (or \emph{non-isotropic}) if and only if $p=\langle v\rangle$ for some $v$ in $V$ with 
$Q(v)\neq 0$.

For lines $\ell$ of $\mathbb{P}(V)$ we can distinguish the following four possibilities:

\begin{enumerate}
\item $\ell$ is \emph{isotropic}, all its points are isotropic;
\item $\ell$ is \emph{hyperbolic}, two of its points are isotropic, all others anisotropic;
\item $\ell$ is \emph{elliptic}, all its points are anisotropic;
\item $\ell$ is \emph{tangent}, one point is isotropic, all others are anisotropic.
\end{enumerate}

Notice that for an anisotropic point $p=\langle v\rangle$ we can have that $Q(v)$ is a square or a non-square in $\mathbb{F}_q$ different from $0$.
So, we can distinguish between points of square-type or non-square type.
Moreover, on a tangent line, all points $p$ are of square type or all are of non-square type.
In particular, the points in the graph $\Gamma_Q$, i.e., a connected  graph on the anisotropic points, where two points are adjacent if 
and only if they span a tangent line, are all of square type or all of  non-square type.

If we replace our form $Q$ by $\alpha Q$, where $\alpha$ is a non-square, we switch the type of the points.
So, square or non-square type is not a really good name, it depends on $Q$.

Now consider the $3$-dimensional case where the form $Q$ is non-degenerate.
There are $q^2+q+1$ points, and we  find that $q+1$ of them are isotropic.
Each pair of isotropic points determines a unique hyperbolic line $H$ and the space $H^\perp$ is a anisotropic point.
Hence there are $\frac{(q+1)q}{2}$ anisotropic points $p$ with $p^\perp$ a hyperbolic line.
We call them the $+$-points.

Let $r$ be the point in the radical of $Q$.
As each isotropic point is on a unique tangent line passing through $r$,
we find $q+1$ tangent lines, and $(q-1)(q+1)=q^2-1$ anisotropic lines outside the radical. 
Each anisotropic point different from $r$ is on a unique hyperbolic line.
The $q$ other lines on such anisotropic line are either hyperbolic or elliptic.
As each hyperbolic line contains two isotropic points, we find $\frac{q}{2}$ hyperbolic lines
on such anisotropic point and $\frac{q}{2}$ elliptic lines.
So, we find that there are $\frac{q(q-1)}{2}$ elliptic lines, and hence 
$\frac{q(q-1)}{2}$ anisotropic points $p$ with $p^\perp$ being elliptic.
These are the $-$-points.
This accounts for all $\frac{(q+1)q}{2}+ \frac{(q-1)q}{2}+q+1=q^2+q+1$ points of the projective plane.
All $+$-points are of square type and all $-$-points of non-square type (or vice versa).

We can generalize this to arbitrary dimension $n\in \mathbb{N}$ of $V$.

If $n$ is odd, there is only one non-degenerate form $Q$ (up to isomorphisms). It is called of $0$ type.
For even $n=2m$, there are two forms $Q$.
We call the form $Q$ of $+$-type if maximal isotropic subspaces have dimension $m$,
and $-$-type if they have dimension $m-1$.

A point $p$ is called of $\epsilon$-type if the hyperplane $p^\perp$ is of $\epsilon$-type, where $\epsilon\in \{+,-,0\}$.

In odd dimension $n=2m+1$, we can split $\mathbb{P}(V)$ into the graphs  $N^{+}O(V,Q)$ and $N^{-}O(V,Q)$, the sets of points $p$ where $p^\perp$ is
of $+$-type and of $-$-type, where two points
are adjacent if and only if they are on a tangent. 
For odd $n$, these two subsets  are the two connected components of the graph on all the anisotropic points.
In case of $n=2m$ being even, all points are of $0$-type, but the graph we are considering has (at least) two isomorphic components denoted by $NO^\pm(V,Q)$
on the points $\langle v\rangle$ with $Q(v)$ being a square or a non-square.
Notice for $n=2$, there are no tangents, so the connected graphs are  just the graphs on a single point.
For $n\geq 4$, these two sets are precisely the two components of the graph on the anisotropic points in $(V,Q)$, where two points are adjacent if and only if they span a tangent.

The condition that for each vertex $v\in \Gamma_Q$ the space $v^\perp$ is of rank at least $2$ is equivalent to $\dim(V)\geq 5$ for
vertices $v$ in $N^+O(V,Q)$, and $\dim(V)\geq 6$ for vertices $v\in NO^\pm(V,Q)$ or $v\in N^-O(V,Q)$.

\section{Tangent geometries}

Consider the situation of the previous section: $V$  is a  vector space over the finite field $\mathbb{F}_q$
where $2\not\mid q$. By $(V,Q)$ we denote the quadratic space on $V$ induced by the non-degenerate quadratic form $Q$.

Then by $(\mathcal{P},\mathcal{L})$ we denote the partial linear space with as point set $\mathcal{P}$ 
the set of anisotropic points
of $\mathbb{P}(V)$  and as lines in $\mathcal{L}$ the tangent lines, or more precisely the sets of anisotropic points contained 
in some tangent line.
As we have seen in the previous section, these geometries are not  connected.
Both $(\mathcal{P},\mathcal{L})$ as well as any connected component  of  $(\mathcal{P},\mathcal{L})$ is called a \emph{tangent geometry} of $(V,Q)$.

A \emph{tangent plane} is the set of anisotropic points in a plane of $\mathbb{P}(V)$ 
on which the form $Q$ has a $2$-dimensional radical. So, the tangent plane consists of all the points  of the plane
outside the radical. Equipped  with its tangent lines, a tangent plane is an affine plane.
Denote $\mathcal{A}$ the set of tangent planes of $(V,Q)$.

We consider the geometry $(\mathcal{P},\mathcal{L},\mathcal{A})$.
If $p\in \mathcal{P}$ is a point, then by $\mathcal{L}_p$ and $\mathcal{A}_p$ we denote the set of lines and tangent planes on $p$, respectively.

\begin{proposition}\label{polar}
Consider the geometry $(\mathcal{P},\mathcal{L},\mathcal{A})$.
Let $p\in \mathcal{P}$ be a point.
Then $(\mathcal{L}_p,\mathcal{A}_p)$ is isomorphic to the polar space of isotropic  points and lines in $p^{\perp}$.

If $\ell\in \mathcal{L}$ is a line on $p$, and $q\in \mathcal{P}$ a point collinear to $p$, then $q$ is collinear with 
 $2$ or all points of $\ell$. 
\end{proposition}

\begin{proof}
Let $p\in \mathbb{P}(V)$ be anisotropic with $p\in \mathcal{P}$. Then a line $\ell$ inside $\mathcal{L}_p$ meets $p^\perp$ in a unique isotropic point, which we denote by
$x_p$.
On the other hand, if $x$ is an isotropic point of $p^\perp$, then the line $\langle p,x\rangle$ is a tangent line. So, $x=x_p$.

If $x_p$ and $y_p$ are two isotropic points inside $p^\perp$, then the line $\langle x_p,y_p\rangle$
is isotropic, if and only if for each $z\in\langle x_p,y_p\rangle$ we find $Q(z)=0$.
But that implies that $\langle x_p,y_p\rangle$ is isotropic, precisely when  $\langle p, x_p,y_p\rangle$
is a $3$-space of $\mathbb{P}(V)$  on which $Q$ has a $2$-dimensional radical, namely $\langle x_p,y_p\rangle$.
The space $\langle p, x_p,y_p\rangle$ meets the set of anisotropic points in a tangent plane.
Hence  $(\mathcal{L}_p,\mathcal{A}_p)$ is isomorphic to the polar space of isotropic  points and lines in $p^{\perp}$.

Now assume that $\ell\in \mathcal{L}_p$, and fix a point  $q\in \mathcal{P}$ different from but collinear to $p$. 
Assume that $q$ is not collinear with all points of $\ell$.
Then the $3$-space  $\langle q,\ell\rangle$ of $\mathbb{P}(V)$ is non-degenerate with respect to $Q$.
It consists of two components  with  $\frac{q(q-1)}{2}$ or  $\frac{q(q+1)}{2}$ points.
In both cases we find $q$ to be collinear with $2$ points of $\ell$.
Indeed, inside this $3$-space, we find that $q^\perp$ is a hyperbolic line with exactly two isotropic points.
The two tangent lines on $q$ inside  this $3$-space both meet $\ell$.
\end{proof}

An \emph{affine polar space} is a polar space from which a geometric hyperplane is removed.
Here a \emph{geometric hyperplane} is a proper subset $H$ of the point set of the polar space which meets each line in one or all points.
The geometry $(\mathcal{P},\mathcal{L},\mathcal{A})$ satisfying the conditions of \cref{polar}
can be obtained from an affine polar space in the following way.

Consider a non-degenerate quadratic space  $(\hat{V},\hat{Q})$ which contains $V$ as a hyperplane with $\hat{Q}_{|V}=Q$.
Let $\hat{P}$ be the set of points in the polar space of $(\hat{V},\hat{Q})$. Then $V$ meets $\hat{P}$ in the geometric hyperplane $\hat{P}\cap \mathbb{P}(V)$.
The affine polar space  obtained by removing the isotropic points of $V$ from the isotropic points in $\hat{V}$ is denoted by 
${AP}(\hat{V},V)$.

Now there is a unique point $r\in \mathbb{P}(\hat{V})$ with $r^\perp$ being the hyperplane $V$. This  point $r$ is anisotropic as $Q$ is non-degenerate.
For any isotropic point $\hat{p}\in \mathbb{P}(\hat{V})$ which is not in $V$ we find that the line $\langle \hat{p},r\rangle$ is  hyperbolic  
and contains a unique point $p$ in $\mathbb{P}(V)$ with $r\perp p$. This point $p$ is also anisotropic.
The map from the affine polar space 
${AP}(\hat{V},V)$ to one of the connected graphs of anisotropic points $\Gamma_Q$ which maps each isotropic point $\hat{p}$ in  ${AP}(\hat{V},V)$
to the anisotropic point $p$ of $\mathbb{P}(V)$
is called a \emph{standard map}.  The image, being the full graph $\Gamma_Q$, is called the standard quotient of the affine polar space.

So, the graph  $\Gamma_Q$ and standard quotient $(\mathcal{P},\mathcal{L}, \mathcal{A})$ 
can be  obtained from ${AP}(\hat{V},V)$, but we can also go in the reverse direction.
Indeed, we can obtain the polar space of $(V,Q)$ from  $(\mathcal{P},\mathcal{L}, \mathcal{A})$.
To find this polar space, we can for each point isotropic point $p$ of $\mathbb{P}({V})$ take the set of those tangent lines in $\mathcal{L}$ which have this point as its missing point. 
The lines on a fix isotropic point $p$ from $V$ form a set of parallel lines, where parallel extends the relation of being parallel in a plane in $\mathcal{A}$.
These parallel classes can  then be identified with  the points of the polar space. 
As a line of this polar space, we can take unions of these parallel classes present in a plane.
In this way we recover the polar space on $V$.
This construction is explained and used in \cite{cuypers,cuyperspasini}.
For more details, we therefore refer the reader to \cite{cuypers,cuyperspasini}.

\begin{corollary}\label{hyperbolic}
Suppose $A\in \mathcal{A}$ is an affine plane and $\ell, m$ two parallel lines in the plane $A$.
Then there is a point $x\in \mathcal{P}$ collinear with just the points of $\ell\cup m$ inside $A$. 
\end{corollary}

\begin{proof}
Let $p$ be a point of $\mathcal{P}$ inside $A$. The lines of $\mathcal{L}$ on $p$ form a polar space isomorphic to the space $p^\perp$.
This is a non-degenerate polar space. Consider a plane $A\in \mathcal{A}$ on $p$.
The lines on $p$ inside $A$ form an isotropic line in the polar space on the lines through $p$.
Fix two lines $\ell$ and $\ell'$ in $A$ on $p$. Then in this polar space  there is a line $n$ on $p$
collinear to $\ell$ but not $\ell'$, and hence there is a point $x'\in n\setminus\{p\}$ collinear with $p$ with 
$\ell\subseteq x^\perp$, but $\ell'\not\subseteq x^\perp$.
As this point $x'$ is collinear with no other line on $p$ inside $A$, we find that each $x'$ is collinear with $2q$ points.
But then these $2q$ points are on two parallel lines inside $A$, one being $\ell$.
Varying the point $x$ over the line through $p$ and $x'$, we find a point $x$ collinear to only the points from $\ell\cup m$ inside $A$.
\end{proof}

\section{Recovering the tangent geometry}
\label{sect:tangent}

As before, fix a  vector space $V$ over a field $\mathbb{F}_q$ with $2\not \mid q>3$, and $Q$ a non-degenerate quadratic form on $V$.
Let $\Gamma_Q$ be one of the connected subgraphs of ${N}O^\epsilon(V,Q)$ or ${N^\epsilon}O(V,Q)$
and suppose $\Gamma$ is a connected graph which is locally isomorphic to $\Gamma_Q$, i.e., for each vertex $v$ of $\Gamma$
the subgraph $\Gamma_v$ on $v$ and all neighbors of $v$ is isomorphic to $\Gamma_{Qv'}$ for some vertex $v'$ of $\Gamma_Q$. 

Let $\mathcal{P}$ be the point set of $\Gamma_Q$ and $\mathcal{P}_\Gamma$ the points of $\Gamma$.
Inside $\Gamma_Q$ we can also find the set $\mathcal{L}$ of tangent lines and $\mathcal{A}$ of tangent planes. 
We will recover similar lines and planes in $\Gamma$.  

The first lemma we prove is important, as it recovers the tangent lines only from local information (i.e., information on $\Gamma_v\simeq \Gamma_{Qv'}$ for some fixed $v\in \Gamma$ and $v'\in \Gamma_Q$).
We recall the following from the previous sections: for $v$ a vertex of $\Gamma_Q$, the geometry of tangent lines and tangent planes on $v$ is a non-degenerate polar space of rank at least $2$, which we denote by $\Pi_v$.

\begin{lemma} 
Fix a vertex $v$ of $\Gamma_Q$. Assume $v^\perp$ is of rank at least $2$.

Let $C$ be a clique of size $|C|\geq 5$ in $\Gamma_Q$ on $v$ such that each vertex $w\not \in C$ adjacent to 
$v$ is adjacent to $2$ or all vertices of $C$.
Then there is a tangent line  $T$ on $v$ with $T=C$.
\end{lemma}

\begin{proof}
Let $C$ be a clique of $\Gamma_Q$ containing a vertex $v$ as in the hypothesis.
For $u$ a point of $\Gamma_Q$ adjacent to but different from $v$,  denote by $T_{v,u}$ the tangent line on $v$ and $u$.

If $C$ is contained in a tangent line, then it equals this tangent line.
Indeed, for any two points of the tangent there is a vertex $x$ adjacent to $v$ which is collinear with 
only these two points and we find  $C$ to be  the tangent line.

If there is a tangent line contained in $C$, then $C$ is equal to this tangent.
Indeed, if $T$ is this tangent and $T\cup \{w\} \subseteq C$ for some point $w$ not in $T$,
then $T$ and $T_{v,w}$ are two points of the polar space induced on the tangents through $v$.
We can find a tangent $T'$ on $v$ with $T$ and $T'$ collinear but $T$ and $T_{v,w}$ not.
But then $T_{v,w}$ contains a point from $\Gamma_v$ which is adjacent to the points of $T$ but not to $w$.

So, suppose $C$ is not contained in a tangent line nor contains a tangent.

Fix now $u,w\in C$ different from $v$ which are on different tangents on $v$ in $\Pi_v$.
Then each vertex $x\neq v$ of a  tangent $S$ on $v$, which is in an affine plane with  $T_{v,u}$ and with $T_{v,w}$,
is adjacent to at least  three points of $C$, the points $v,u,w$, and thus to all $q$ points of $C$.
In particular, for any $y\in C$ different from $v$ we find $T_{v,y}$ to be collinear to $T_{v,x}$ and hence $y$ to be adjacent 
all points of each such  $S$.
Moreover,  $T_{v,y}$ is contained in $\{T_{v,u}, T_{v,w}\}^{\perp\perp}$ inside the polar space $\Pi_v$ on the tangents on $v$.

If  $T_{v,u}\perp T_{v,w}$, then we find each such $T_{v,y}$ to be contained in the polar line on $T_{v,u}$ and $T_{v,w}$
of $\Pi_v$ which carries the structure of an affine plane. 
In particular, $C$ is contained in the affine plane on $v$ generated by $T_{v,u}$ and $T_{v,w}$.
This affine plane contains three non-collinear points inside $C$, say $v=t_1$, $t_2$ and $t_3$.
Then fix a point $x$ adjacent to these three points but not all of the plane. 
Such point does exist in the graph $\Gamma_Q$, see \cref{hyperbolic}.
Then as all points of $C$ are collinear to $x$, the clique $C$ is contained in the two parallel lines collinear to $x$.
Taking a different point $x'$ which is also adjacent to the three $t_1$, $t_2$ and $t_3$, but with 
the points of the affine plane adjacent to $x'$ being in two parallel lines from a different parallel classes, again see \cref{hyperbolic},
shows that $C$ can only have size smaller than $5$, leading to a contradiction.
Hence $C$ equals $S$ for some tangent line $S$ of $\Gamma$ on $v$, in the case that  $T_{v,u}\perp T_{v,w}$.

Now consider the case that $T_{v,u}\not \perp T_{v,w}$. Then we find $T_{v,y}$ to be contained in the hyperbolic line on $T_{v,u}$ and $T_{v,w}$
of $\Pi_v$ which just contains the two tangent lines $T_{v,u}$ and $T_{v,w}$. So $C\subseteq T_{v,u}\cup T_{v,w}$.

Since $C$ meets both $T_{v,u}$ and $T_{v,w}$ in at least two points, and $C$ has at least $5$ points, we can assume that
it meets $T_{v,u}$ or $T_{v,w}$ in at least $3$ points.
Take a vertex $x$ with $x$ adjacent to $v$ and  all vertices of $T_{v,u}$ but only to two points of $T_{v,w}$.
Then $x$ is adjacent to at least $3$ and thus all points of $C$. So, $C$ meets $T_{v,w}$ in just the two points $v$ and $w$.
Now picking $x$ to be a point adjacent to $v,w$ and $u$, and all points of $T_{v,w}$ we find that  $C$ meets $T_{v,u}$ in just the two points $v$ and $u$.
But then $C$ is contained in $\{v,w,u\}$ and $|C|\leq 3$ contradicting $|C|\geq 5$.
This contradiction implies that in this case $C$ is a tangent line.

As we only used neighbors of $v$ in the above arguments, this implies the lemma.
\end{proof}

\begin{remark}
The above lemma is the only place in this paper where the case with $q=3$ does not hold.
The case that $q=3$ fails in the above lemma, as there are cliques of size $3$ in $\Gamma_Q$ which are not inside a tangent line.
In case $(V,Q)$ is a quadratic space over $\mathbb{F}_3$, there is a local characterization by Pasechnik \cite{pasechnik} of the graphs 
on anisotropic points where two points are adjacent if and only if they span a hyperbolic line. 
\end{remark}

%\bigskip

As the above result is only proven by arguing with points collinear to $v$, 
we find that $\Gamma$ has the same set of tangent lines and planes as the graph $\Gamma_Q$ at this point $v$.
That also implies that  $\Gamma$ has a set of lines, which behave like the tangent lines of $\Gamma_Q$ at each point $v$.
Indeed, for each point $v$ of $\Gamma$ we find a line being a clique of size $q$ on $v$ such that any point which is adjacent to $v$ is adjacent to
$2$ or all points of that clique.
Denote this set of lines  by $\mathcal{L}_{\Gamma v}$. Now by $\mathcal{L}_\Gamma$ we denote the union of all cliques of size $q$ in $\Gamma$ which are inside some $\Gamma_v$ where $v$ is a vertex of $\Gamma$. 

\begin{lemma}
The geometry $(\mathcal{P}_\Gamma,\mathcal{L}_\Gamma)$ is a partial linear space.
\end{lemma}

\begin{proof}
Let $\ell$ be a line in $\mathcal{L}_{\Gamma,v}$ and suppose $w$ is a point in $\ell$ different from $v$.
We have to show that the line $\ell$ is the unique  line of $\Gamma_w$ containing $v$.
The line $\ell$ equals $T_{v,w}$, we have to show that it also equals $T_{w,v}$.

The line $T_{v,w}$ is defined as being a clique $C$ of size at least $5$ in $\Gamma$ containing the points $v$ and $w$ such that each neighbor $x\not \in C$ of $v$
is adjacent to   $2$ or to all points of $C$.
Inside the polar space of the tangents on $v$ in $\Gamma_{Q v}$, we can also deduce that it is the intersection of the 
sets consisting of all $y$ adjacent to $x$, where $x$ runs through the points
adjacent to all vertices of $C$.
%Indeed, $$C=\bigcap_{x\in C^\perp} x^\perp.$$
However, this latter definition does not depend on the points $v$ and $w$.

So, $C$ can also be defined in this way. In particular, we find $T_{v,w}=T_{w,v}$.
\end{proof}

We can also define affine planes on $\Gamma$. 
Let $\ell,m$ be tangent lines in $\Gamma_v$ on $v$ which are inside a clique of $\Gamma$.
Then these two lines are collinear in the polar space on the lines of $\Gamma$ through $v$.
They determine a unique polar line on the points $\ell$ and $m$ with $q+1$ points inside $\Gamma_v$, which we denote by $A_{v,\ell,m}$. 

\begin{lemma}\label{planes}
Let $\ell$ and $m$ be two lines in $\mathcal{L}_{\Gamma,v}$ on $v$ which are inside a clique $A_{v,\ell,m}$ on  $q+1$ lines on $v$. 
Then the  $1+(q+1)(q-1)=q^2$ points inside $A_{v,\ell,m}$  form an affine plane in $(\mathcal{P}_\Gamma,\mathcal{L}_\Gamma)$.
\end{lemma}

\begin{proof}
Let $\pi=A_{v,\ell,m}$ be the set of $q^2$ points. This set does have the structure of an affine plane on $v$.
Indeed, a point $x$ adjacent to $v$ is adjacent to all or $2q$ points in $\pi$ which are on two non-intersecting tangent lines of the plane,
one tangent line on $v$ and one not on $v$ but parallel.
By varying the point $x$ over the points that have $2q$ neighbors in $\pi$, we find that, seen from $v$, the lines induced by these points provide the structure of an affine plane on the $q^2$ points.
We indicate this structure by $\pi_v$.

If $w\in \pi$ is a point different from $v$, then $w$ also determines an affine structure on $\pi$,
which we will show to be the same as the one induced by $v$.

Clearly, for a point $u$ adjacent  to both $v$ and $w$, the set  $u^\perp$ intersects $\pi$ in at least a line on $v$ and the parallel line on $w$ within the structure of 
the affine plane $\pi_v$.
By only considering $u$ such that $u^\perp$ meets the line $k=T_{v,w}=T_{w,v}$ on $v$ and $w$ just in these two points, this defines all lines in $\pi_v$ on $w$.
These lines are all inside some $u^\perp\cap \pi$ but are not containing a point of a line of $\pi_v$ through $v$ inside this set.
We will show that each such line is also a line in $\mathcal{L}_{\Gamma w}$.

Each point $y$ which is adjacent to $w$ has $y^\perp$ meeting $k$ in two or all points.
If we fix a $y$ which is adjacent to both $v$ and $w$, we find that $y^\perp$ contains (at least) two lines
of $\pi_v$, one on $v$ and one on $w$, as above.

If $y$ is not adjacent to $v$, but to $w$, then the line in $\mathcal{L}_{\Gamma w}$ on $w$ and $y$ contains a point $y'$ adjacent to $v$.
Notice that $y'^\perp$ meets $\pi$ in a clique on $w$ which is, as we saw above, the union of two  lines in $\pi_v$. One on $v$ and one on $w$.
In particular, $y^\perp \cap y'^\perp\cap \pi$ is the clique on $q$ points on $w$ inside $\pi$ which is a line on $w$ of $\pi_v$.
Varying $y$ we find all lines on $w$ in $\pi_v$ to be lines of $\mathcal{L}_{\Gamma w}$, and varying 
$w$ we find $\pi$ to be the affine plane $\pi_v$. By connectedness of $\Gamma$ we find all such planes to be affine.
\end{proof}

Denote by $\mathcal{A}_\Gamma$ the set of all these affine planes we obtain by the above \cref{planes}.
Then the geometry $(\mathcal{P}_\Gamma,\mathcal{L}_\Gamma,\mathcal{A}_\Gamma)$ is locally a non-degenerate polar space of rank at least $2$.
But these geometries are characterized.
Indeed, we obtain the following result which can be found in \cite{cuypers} for the rank of the local polar space at a point equal to $2$ (in which case the geometry is finite) and 
in \cite{cuyperspasini} for the case that this rank is at least $3$.

\begin{theorem}
The geometry $(\mathcal{P}_\Gamma,\mathcal{L}_\Gamma,\mathcal{A}_\Gamma)$ is a  connected geometry with a set $\mathcal{A}_\Gamma$ of affine planes, such that for all points $p\in \mathcal{P}_\Gamma$ we have
$(\mathcal{L}_{\Gamma p},\mathcal{A}_{\Gamma p})$ a non-degenerate polar space of orthogonal type with lines of size $q+1$, where $q$ is odd, and rank at least $2$.

Moreover, as each point is collinear with $0$, $2$ or all points of a line, $(\mathcal{P}_\Gamma,\mathcal{L}_\Gamma,\mathcal{A}_\Gamma)$ is a standard quotient of an affine orthogonal polar space.
\end{theorem}

So, to prove our main theorem, \cref{main}, it remains to check which standard quotients of affine orthogonal polar spaces
we have to deal with.
However, as inside $\Gamma$ each clique $C$, which is a tangent line, does satisfy the rule that each point $x$ is adjacent to $0,2$ or all points
of $C$, we do notice that $\Gamma$ is a standard quotient of the affine polar space obtained by calling two points $y$ and $z$
in relation if and only if $\langle y,z\rangle$ contains the point $r$ with  $r^\perp=H$.
The geometric hyperplane which we have removed is then  a non-degenerate geometric  hyperplane, as each point $p$ of the polar space outside $H$ is (in the quotient) collinear to $2$ points of a line $\ell$, which contains a point but not all in $p^\perp$, and the number of elements in a hyperbolic line of the local polar space $(\mathcal{L}_{\Gamma p},\mathcal{A}_{\Gamma p})$ is also $2$. See \cite{cuypers} and \cite{cuyperspasini}.
So, we do find that $\Gamma$ is the collinearity graph of this standard quotient of an affine orthogonal polar space obtained by removing a non-degenerate geometric hyperplane. The graph $\Gamma$ is isomorphic to $\Gamma_Q$. We have proven \cref{main}:

\begin{theorem}
Let $\Gamma$ be a connected graph locally isomorphic to the graph $\Gamma_Q$.
Assume that for each $v\in \Gamma_Q$ the rank of $v^\perp$ is at least $2$ and $2\not \mid q>3$.
Then $\Gamma$ is isomorphic to $\Gamma_Q$.
\end{theorem}

\section{Automorphisms of $\Gamma_Q$}

Consider a  non-degenerate orthogonal space $(V,Q)$ over a field $\mathbb{F}_q$ with $q>3$ odd. 
In the previous section we proved that the graph $\Gamma_Q$ with point set $\mathcal{P}$ contains, under the restrictions that for all anisotropic  $v\in V$ the space $v^\perp$
has rank at least $2$, a unique collection $\mathcal{L}$ of lines of size $q$ as well as
a unique collection $\mathcal{A}$ of affine planes such that the geometry $(\mathcal{P},\mathcal{L},\mathcal{A})$
is a standard quotient of an affine polar space of rank $r\geq 2$.
It has been shown in \cite{cuypers} and \cite{cuyperspasini} that these quotients also carry the structure of the polar space of $(V,Q)$, with $V$ of dimension at least $5$.
Indeed, the maximal standard quotient of the affine polar space is an orthogonal polar space isomorphic to the polar space of $(V,Q)$.
Its points are the equivalence classes of extended parallel lines with as polar lines the equivalence classes of the lines in some affine plane.
This implies that the automorphism group of the graph $\Gamma_Q$ is contained in the automorphism group of this polar space.
Knowing the automorphism group of the polar space, see \cite{BC}, we obtain:

\begin{lemma}
Suppose for each $v\in \Gamma_Q$ the rank of $v^\perp$  is at least $2$ and $2\not\mid q>3$.
Then $\mathrm{Aut}(\Gamma_Q)$ is contained in $P\Gamma O(V,Q)$. 
\end{lemma}

As the group $P\Gamma O(V,Q)$ does act as a group of automorphisms of the graph $\Gamma_Q$ we even obtain  the following:

\begin{theorem}
Suppose for each $v\in \Gamma_Q$ the rank of $v^\perp$  
is  at least $2$ and $2\not\mid q>3$.
Then the  automorphism group of the graph $\Gamma_Q$ is isomorphic to $P\Gamma O(V,Q)$. 
\end{theorem}

This result is equivalent to \cref{aut}.

\begin{remark}
The condition that the rank of $v^\perp$ for a point $v\in V$ with $Q(v)\neq 0$ is at least $2$ is a necessary condition.
For example, in case the dimension of $V$ equals $3$, then we find that the automorphism group of a connected component
$N^\epsilon O(V,Q)$ to contain the  symmetric group $\mathrm{Sym}_{q+1}$ acting on the $\frac{(q+1)q}{2}$ points in case $\epsilon=+$ and
$\mathrm{Sym}_{q}$ acting  on  the $\frac{(q-1)q}{2}$ points in case  $\epsilon=-$. In both cases, the action of the symmetric group $\mathrm{Sym}_{m}$ is the action on the pairs of elements from $\{1,\dots, m\}$.
\end{remark}

\begin{remark}
The result \cref{aut} has also been obtained by our  student Maud Bokhoven  in her Bachelor Project  \cite{maud}.
\end{remark}

\newpage 

\bibliographystyle{abbrv}
\bibliography{ortho.bib}

\vspace{1cm}
\bigskip

\begin{minipage}{10cm}
Hans Cuypers\\
Department of Mathematics and Computer Science\\
Eindhoven University of Technology\\
P.O. Box 513, 5600 MB Eindhoven\\
The Netherlands\\
email: f.g.m.t.cuypers@tue.nl\\
\end{minipage}
\end{document}